\documentclass[psreqno,reqno,11pt]{amsart}
\usepackage{eurosym}
\usepackage{amssymb}
\usepackage{tikz}
\usepackage{amsmath}
\usepackage{tikz}
\usepackage{pgflibraryplotmarks}
\usepackage{wasysym}
\usepackage{stmaryrd}
\usepackage[english]{babel}

\usepackage{hyperref}

\theoremstyle{plain}
\newtheorem{theorem}{Theorem}[section]
\newtheorem{proposition}[theorem]{Proposition}
\newtheorem{lemma}[theorem]{Lemma}
\newtheorem{corollary}[theorem]{Corollary}

\theoremstyle{definition}
\newtheorem{definition}[theorem]{Definition}
\newtheorem{example}[theorem]{Example}

\newtheorem{notation}[theorem]{Notation}

\begin{document}
\def\N{\mathbb{N}}
\def\Z{\mathbb{Z}}
\def\R{\mathbb{R}}
\def\P{\mathbb{P}}
\def\w{\omega}
\def\H{\mathfrak{H}}
\def\G{\mathfrak{G}}
\def\Q{\mathfrak{Q}}

\title[Harmonic mean closed vector lattices]{
Harmonic mean closed vector lattices}
\author{C. Schwanke}
\address{Department of Mathematics and Applied Mathematics, University of Pretoria, Private Bag X20, Hatfield 0028, South Africa and Unit for BMI, North-West University, Private Bag X6001, Potchefstroom, 2520, South Africa}
\email{cmschwanke26@gmail.com}
\date{\today}

\begin{abstract}
We introduce the notion of the harmonic mean on certain Archimedean vector lattices, which we call harmonic mean closed. Using this concept, we provide a new characterization of positive orthogonally additive polynomials from a harmonic mean closed Archimedean vector lattice to a uniformly complete Archimedean vector lattice space entirely in terms of harmonic means. This result yields polynomial versions of the QM-AM-GM-HM inequalities on Archimedean vector lattices where the only additional requirement on the domain space is that it is harmonic mean closed.
\end{abstract}

\maketitle
\section{Introduction}\label{S:intro}

For $n\in\N$, the quadratic mean $\mathfrak{Q}_n$ and the geometric mean $\mathfrak{G}_n$ are given by
\[
\mathfrak{Q}_n(x_1,\dots,x_n)=\sqrt{\dfrac{1}{n}\sum_{k=1}^{n}x_k^2}\quad (x_1,\dots,x_n\in\mathbb{R})
\]
and
\[
\mathfrak{G}_{n}(x_1,\dots,x_n)=\sqrt[n]{\prod_{k=1}^{n}|x_k|}\quad (x_1,\dots,x_n\in\mathbb{R}).
\]

These means have enjoyed extensive study in the setting of Archimedean vector lattices recently, see e.g. \cite{Az, AzBoBus, BusSch, BusSch4, Kusa, Sch}. Indeed, these means can be defined in uniformly complete vector lattices using the Archimedean vector lattice functional calculus developed in \cite{BusdPvR}. Of particular interest to this paper, it is proven in \cite{Kusa} that if (i) $E$ is a uniformly complete vector lattice, (ii) $Y$ is a separated convex bornological space, and (iii) $P_1,P_2\colon E\to Y$ are bounded orthogonally additive polynomials, with $P_1$ $2$-homogeneous and $P_2$ $n$-homogeneous ($n\in\mathbb{N}$), and with $P_2$ having unique corresponding symmetric $n$-linear map $\check{P}_2$, then the following hold:
\begin{equation}\label{eq:RMP}
	P_1(\mathfrak{Q}_{n}(f_{1},\dots,f_{n}))=\dfrac{1}{n}\sum_{k=1}^nP_1(f_{k})\quad (f_{1},\dots,f_{n}\in E^+)
\end{equation}
and
\begin{equation}\label{eq:GM}
	P_2(\mathfrak{G}_{n}(f_{1},\dots,f_{n}))=\check{P}_2(f_{1},\dots,f_{n})\quad (f_{1},\dots,f_{n}\in E^+).
\end{equation}

The intimate relationship between means and orthogonally additive polynomials was further explored in \cite[Theorems 2.3\&2.4]{BusSch4} and \cite[Theorem~2.3]{Sch}, where it was shown that the identities \eqref{eq:RMP} and \eqref{eq:GM} in fact characterize bounded orthogonally additive polynomials $P\colon E\to Y$.

The results mentioned above illustrate how the quadratic and geometric means play intriguing roles in vector lattice theory. It is perhaps puzzling then why another well-known mean, the harmonic mean, has yet to be explored in this setting. Our primary result, Theorem~\ref{T:HM}, illustrates the harmonic mean's relevance in vector lattice theory by providing a novel characterization of positive orthogonally additive polynomials similar to \eqref{eq:RMP} and \eqref{eq:GM} above, but solely in terms of the \textit{harmonic mean}. The usefulness of this depiction of positive orthogonally additive polynomials can be seen in Theorem~\ref{T: HM-GM-AM-QM}, where we utilize Theorem~\ref{T:HM} in proving polynomial versions of the QM-AM-GM-HM inequalities on Archimedean vector lattices, where the only additional assumption required on the domain space is that it is harmonic mean closed--geometric mean completeness nor quadratic mean closedness are needed suppositions.

We proceed with some preliminaries.

\section{Preliminaries}\label{S: prelims}

	The reader is referred to the standard texts (e.g. \cite{AB,LuxZan1, LilZan, Zan2}) for any unexplained terminology or basic theory regarding vector lattices. Throughout, $\mathbb{R}$ is used for the real numbers, and the set of strictly positive integers is denoted by $\mathbb{N}$.

\begin{definition}\label{D: OVS & VL}
A real vector space equipped with a partial ordering $\leq$ is termed an \textit{ordered vector space} if its algebraic and order structures are compatible in the following sense:
\begin{itemize}
	\item[(i)] $f\leq g$ implies $f+h\leq g+h\quad (f,g,h\in E)$, and
	\item[(ii)] $f\geq 0$ implies $\lambda f\geq 0\quad (f\in E, \lambda\in[0,\infty))$.
\end{itemize}
The \textit{positive cone} of an ordered vector space $E$ is the set $E^+:=\{f\in E\ :\ f\geq 0\}$. Additionally, we say $E$ is \textit{Archimedean} if, for each $x \in E$ and every $y \in E^+$,
\[
nx \leq y \quad \text{for all}\ n \in\N\quad \text{implies}\quad x \leq 0.
\]

An ordered vector space $E$ that is also a lattice under its partial ordering; i.e. for each $f,g\in E$, both $f\vee g:=\sup\{f,g\}$ and $f\wedge g:=\inf\{f,g\}$ exist in $E$, is called a \textit{vector lattice}. In particular, if $E$ is a vector lattice and $f\in E$, the \textit{absolute value} of $f$ always exists in $E$ and is defined by $|f|:=f\vee(-f)$. Moreover, if $f,g\in E$ satisfy $|f|\wedge|g|=0$, then we say that $f$ and $g$ are \textit{disjoint} and write $f\perp g$.
\end{definition}

Also important to the theory of this paper are projection bands and band projections.

\begin{definition}
Given a vector lattice $E$, we call a subset $A$ of $E$ \textit{solid} if whenever $f\in E$, $g\in A$, and $|f|\leq|g|$, we have $f\in A$. A solid vector subspace of $E$ is called an (order) \textit{ideal}. An ideal $B$ of $E$ with the property that $\sup B\in B$ whenever $\sup B$ exists in $E$ is called a \textit{band}. Any band $B$ of $E$ for which $E=B\oplus B^d$, where
\[
B^d:=\{f\in E\ :\ f\perp b\quad \text{for all}\quad b\in B\},
\]
is called a \textit{projection band}. Moreover, the \textit{band projection} onto $B$ is the map $\mathbb{P\colon E\to B}$ defined by $\mathbb{P}(f)=u$, where $f=u+v$ is the unique decomposition of $f$ where $u\in B$ and $v\in B^d$. Given $f\in E$, the smallest band in $E$ that contains $f$ is called the \textit{principal band generated by} $f$ and is denoted by $B_f$. We will always denote its corresponding band projection by $\P_f$. A vector lattice for which every band is a projection band is said to have the \textit{projection property}, and the term \textit{principal projection property} refers to the property of having every principal band $B_f$ in a vector lattice being a projection band.
\end{definition}

We next remind the reader of a class of functions which are central to the research in this paper.

\begin{definition}
	Given $n\in\N$, a function $h\colon\R^n\to\R$ is called \textit{positively homogeneous} if for all $\lambda\in[0,\infty)$ and every $x_1,\dots,x_n\in\R$,
	\[
	h(\lambda x_1,\dots,\lambda x_n)=\lambda h(x_1,\dots,x_n).
	\]
\end{definition}

Next we provide some notation that was used in \cite{BusdPvR} to develop the Archimedean vector lattice functional calculus. Given vector lattices $E$ and $F$, a linear map $\omega\colon E\to F$ is called a \textit{vector lattice homomorphism} if, for every $f,g\in E$ we have
\[
\omega(f\vee g)=\omega(f)\vee \omega(g)\qquad \text{and}\qquad \omega(f\wedge g)=\omega(f)\wedge \omega(g).
\]

\begin{notation}
	Given $n\in\N$, an Archimedean vector lattice $E$, and $f_1,\dots,f_n\in E$, we denote by $\langle f_1,\dots,f_n\rangle$ the vector sublattice of $E$ generated by $\{f_1,\dots,f_n\}$, and the set of all real-valued vector lattice homomorphisms defined on $E$ is written as $H(E)$.
\end{notation}

We next present the Archimedean vector lattice functional calculus developed by Buskes, de Pagter, and van Rooij in \cite{BusdPvR}.

\begin{definition}\cite[Definition~3.1]{BusdPvR}\label{D: fun cal}
	Fix $n\in\N$. Consider an Archimedean vector lattice $E$, and let $f_1,\dots,f_n\in E$. Suppose $h\colon\R^n\to\R$ is a continuous positively homogeneous function. If there exists $g\in E$ such that for every $\omega\in H(\langle f_1,\dots,f_n,g\rangle)$,
	\[
	h\bigl(\omega(f_1),\dots,\omega(f_n)\bigr)=\omega(g),
	\]
	then we define $h(f_1,\dots,f_n):=g$. We refer to this definition as the \textit{Archimedean vector lattice functional calculus} in this paper.
\end{definition}

By \cite[Lemma~3.2]{BusdPvR}, if such a $g$ exists in Definition~\ref{D: fun cal}, then it is unique. Moreover, by \cite[Theorem~3.7]{BusdPvR}, the Archimedean functional calculus can always be defined, for any continuous positively homogeneous function $h$, on uniformly complete Archimedean vector lattices. For completeness, we give the relevant definitions below.

\begin{definition}
Let $E$ be an Archimedean vector lattice. We say a sequence $(f_n)$ in $E$ is \textit{relatively uniformly Cauchy}, respectively, \textit{converges relatively uniformly} to $f\in E$, if there exists $u\in E^+$ such that, for every $\epsilon>0$, there exists $N\in\N$ such that, for all $m,n\geq N$, we have
\[
|f_m-f_n|\leq \epsilon u,\qquad \text{respectively},\qquad |f_n-f|\leq\epsilon u.
\]
An Archimedean vector lattice $E$ is called \textit{uniformly complete} if every relatively uniformly Cauchy sequence in $E$ converges relatively uniformly to an element of $E$.
\end{definition}

More generally, given a continuous positively homogeneous function $h\colon\R^n\to\R$, an Archimedean vector lattice $E$ is called $h$-\textit{complete} if for every $f_1,\dots,f_n\in E$, there exists $b\in E$ such that $h(f_1,\dots,f_n)=b$ \cite[Definition~3.2]{BusSch}.

If a continuous positively homogeneous function $h$ is either convex or concave, we have the following useful explicit formulas for elements defined via the Archimedean vector lattice functional calculus with respect to $h$.

\begin{theorem}\cite[Theorem~3.7(1)]{BusSch}\label{T: convex & concave}
 Let $E$ be an Archimedean vector lattice, and fix $n\in\N$. Assume $h\colon \R^n\to\R$ is continuous and positively homogeneous, and let $f_{1},...,f_{n}\in E^{+}$. Furthermore, set
 \[
 \Delta_{h}:=\lbrace\text{\textit{\textbf{c}}}\in(\mathbb{R}^{+})^{n}:h\ \text{is differentiable at}\ \text{\textit{\textbf{c}}}\ \text{and}\ \|\text{\textit{\textbf{c}}}\|=1\rbrace.
 \]
\begin{itemize}
	\item[(i)] If $h$ is convex on $(\mathbb{R}^{+})^{n}$, then $h(f_{1},...,f_{n})=b$ for some $b\in E$ if and only if
	\[
	b=\sup\left\{\sum\limits_{k=1}^{n}\frac{\partial h(\text{\textit{\textbf{c}}})}{\partial x_{k}}f_{k}\ :\ \text{\textit{\textbf{c}}}\in\Delta_{h}\right\}.
	\]
	\item[(ii)] If $h$ is concave on $(\mathbb{R}^{+})^{n}$, then $h(f_{1},...,f_{n})=b$ for some $b\in E$ if and only if
	\[
	b=\inf\left\{\sum\limits_{k=1}^{n}\frac{\partial h(\text{\textit{\textbf{c}}})}{\partial x_{k}}f_{k}\ :\ \text{\textit{\textbf{c}}}\in\Delta_{h}\right\}.
	\]
\end{itemize}
\end{theorem}

In addition to the formulas in Theorem~\ref{T: convex & concave}, sometimes continuous positively homogeneous functions $h$ can be explicitly defined via the multiplication structure in semiprime Archimedean $f$-algebras. We provide some examples in Proposition~\ref{P: sm gm} after giving the necessary definitions.

\begin{definition}
A vector space $E$ is called an \textit{associative algebra} if there exists a map
\[
E\times E\ni(f,g)\mapsto fg\in E,
\]
called a \textit{multiplication} on $E$, for which the following hold.
\begin{itemize}
	\item[(i)] $f(g+h)=fg+fh\quad (f,g,h\in E)$.
	\item[(ii)] $(f+g)h=fh+gh\quad (f,g,h\in E)$.
	\item[(iii)] $(fg)h=f(gh)\quad (f,g,h\in E)$.
	\item[(iv)] $(\alpha f)(\beta g)=(\alpha\beta)fg\quad (f,g\in E, \alpha,\beta\in\R)$.
\end{itemize}

When an associative algebra $E$ is a vector lattice for which
\begin{itemize}
	\item[(v)] 	$fg\in E^+$ for every $f,g\in E^+$,
\end{itemize}
we call $E$ an $\ell$-\textit{algebra}. An $f$-\textit{algebra} $E$ is an $\ell$-algebra which satisfies the following additional properties.
\begin{itemize}
	\item[(vi)] If $f,g\in E$ satisfy $f\wedge g=0$, and $c\in E^+$, then $(cf)\wedge g=0$.
	\item[(vii)] If $f,g\in E$ are such that $f\wedge g=0$, and $c\in E^+$, then $(fc)\wedge g=0$.
\end{itemize}
We call an $f$-algebra \textit{semiprime} if it has no nonzero nilpotents.
\end{definition}

Given an Archimedean $f$-algebra $E$, $f\in E^+$, and $n\in\N$, if there exists a unique $r\in E^+$ such that $r^n=f$, we call $r$ the \textit{$n$th root} of $f$ and denote it by $r:=\sqrt[n]{f}$. As usual, we in this case set $\sqrt{f}:=\sqrt[2]{f}$. 

Following \cite[Theorem~2.21 and Proposition 2.24]{Az}, we get the following $f$-algebra formulas for quadratic mean and the geometric mean.

\begin{proposition}\label{P: sm gm}
Let $E$ be a uniformly complete Archimedean semiprime $f$-algebra, and fix $n\in\N$. For $f_1,\dots,f_n\in E^+$, we have
\begin{itemize}
	\item[(i)] $\displaystyle\mathfrak{Q}_n(f_1,\dots,f_n)=\sqrt{\dfrac{1}{n}\sum_{k=1}^{n}f_k^2}$, and
	\smallskip
	\item[(ii)] $\displaystyle\mathfrak{G}_{n}(f_1,\dots,f_n)=\sqrt[n]{\prod_{k=1}^{n}f_k}$.
\end{itemize}
\end{proposition}

A similar expression for the harmonic mean requires more structure than a uniformly complete Archimedean semiprime $f$-algebra. This requirements leads us to our next definition.

\begin{definition}
A vector lattice $E$ is called \textit{Dedekind complete} if every nonempty subset which is bounded above has a supremum. If every nonempty subset consisting of pairwise disjoint elements possesses a supremum in $E$, then we say that $E$ is \textit{laterally complete}. A \textit{universally complete} vector lattice is a vector lattice that is both Dedekind complete and laterally complete.
\end{definition}

A key observation is that every universally complete vector lattice is an Archimedean unital $f$-algebra (see e.g. \cite[Theorem~15.22]{dJvR}) and therefore semiprime \cite[Theorem~142.5]{Zan2}. Furthermore, every Dedekind complete, and hence universally complete, vector lattice is uniformly complete \cite[Theorem~12.8]{LilZan}, has the projection property \cite[Theorem~1.42]{AB}, and is Archimedean \cite[Theorem~12.3]{LilZan}.

Though not every Archimedean vector lattice is universally complete, there exists, for each such space $E$, a unique (up to vector lattice isomorphism; i.e. bijective vector lattice homomorphism) universally complete vector lattice $E^u$, called the \textit{universal completion} of $E$, for which $E$ is an order dense vector sublattice of $E^u$ (see e.g. \cite[Section~2.3]{AB}).

We proceed by recording some fundamental properties of a universally complete vector lattices $E$. Statements (i)--(iv) can be found in \cite[Section 142]{Zan2}, and (v) is a consequence of (iv), which follows from \cite[Theorem 2.37]{AB}.

For all $f,g\in E$ and any band projection $\P$ on $E$, we have

\begin{itemize}
	\item[(i)] $fg=gf$.
	\item[(ii)] $|fg|=|f||g|$.
	\item[(iii)] $f\perp g$ if and only if $fg=0$.
	\item[(iv)] If $f$ is invertible and positive then $f^{-1}$ is positive.
	\item[(iv)] $\P(fg)=f\P(g)$.
	\item[(v)] $\P(fg)=\P(f)\P(g)$.
\end{itemize}

Given a vector lattice $E$, we call $w\in E$ a \textit{weak order unit} if $w\in E^+$ and $B_w=E$. Importantly for this paper, a positive element in a universally complete vector lattice is a weak order unit if and only if it is multiplicatively invertible (see e.g. \cite[Remark~3.3]{RoeSch}). This fact, together with the observation that every band in a universally complete vector lattice is itself a universally complete vector lattice, lead to our next definition.

\begin{definition}\label{D: vn inverse}
Given a universally complete vector lattice $E$ and $f\in E\setminus\{0\}$, we denote by $f^\ast$ the multiplicative inverse of $f$ in the band $B_f$ (where $f$ is a weak order unit and hence multiplicatively invertible). Also, we set $0^\ast:=0$. In particular, for $x\in\R$, we have
\[
x^\ast=\begin{cases}
x^{-1} & \text{if}\ x\neq 0\\
0 & \text{if}\ x=0
\end{cases}.
\]
\end{definition}

Let $E$ and $Y$ be vector spaces, and let $n\in\mathbb{N}$. A map $P\colon E\to Y$ is called an $n$-\textit{homogeneous polynomial} if there exists a (unique) symmetric $n$-linear map $\check{P}\colon E^n\to V$ such that $P(f)=\check{P}(f,\dots,f)\ (f\in E)$. (We denote the symmetric $n$-linear map associated with an $n$-homogeneous polynomial $P$ by $\check{P}$ throughout.) If $E$ and $Y$ are ordered vector spaces, we also say that $P$ is \textit{positive} whenever $\check{P}$ is positive, meaning that $\check{P}(f_1,\dots,f_n)\in Y^+$ whenever $f_1,\dots,f_n\in E^+$. We also say that $P$ is \textit{monotone on the positive cone} if, for every $f,g\in E^+$ with $f\leq g$, we have $P(f)\leq P(g)$. The next result from \cite{Loane1} states that every positive homogeneous polynomial is monotone on the positive cone.

\begin{proposition}\cite[Proposition~5]{Loane1}\label{pos imps pos on mon cone}
Let $E$ and $Y$ be ordered vector spaces, and fix $n\in\N$. If $P\colon E\to Y$ is a positive $n$-homogeneous polynomial, then $P$ is monotone on the positive cone.	
\end{proposition}

Moreover, if $E$ is a vector lattice, and $V$ is any vector space, an $n$-homogeneous polynomial $P\colon E\to V$ is said to be \textit{orthogonally additive} if
\[
P(f+g)=P(f)+P(g)
\]
holds whenever $f,g\in E$ are disjoint. More generally, if 
\[
P(f+g)=P(f)+P(g)
\]
holds for all $f,g\in E^+$ with $f\perp g$, then we say that $P$ is \textit{positively orthogonally additive}.

Given ordered vector spaces $E$ and $V$, with $E$ having a generating positive cone (meaning for all $f\in E$, there exist $f_0,f_1\in E^+$ for which $f=f_0-f_1$), and a fixed $n\in\N$, any positive $n$-homogeneous polynomial $P\colon E\to V$ is \textit{order bounded}, meaning that $P(B)$ is an order bounded set in $V$ whenever $B$ is an order bounded set in $E$. By an \textit{order bounded} set, say in $E$, we mean $B\subseteq E$ for which there exists $a,b\in E$ such that, for all $x\in B$, we have $a\leq x\leq b$. To verify our claim about positive polynomials being order bounded, if $a\leq x\leq b$ holds in $E$, then defining $u:=x-a$ gives us (i) $x=a+u$, (ii) $u\in E^+$, and (iii) $u\leq b-a$. Since $E$ has generating positive cone, there exist $a_0,a_1\in E^+$ such that $a=a_0-a_1$. Expanding $\check{P}$ via $n$-linearity gives us
\begin{align*}
P(x)&=P(a+u)\\
&=P(a_0-a_1+u)\\
&=\sum_{i_1,\dots,i_{k}\in\{0,1\}}(-1)^{\sum_{j=1}^{k}i_j}T(a_{i_1},\dots,a_{i_k},\underbrace{u,\dots,u}_{n-k\ \text{copies}})\\
&=\sum_{\substack{i_1,\dots,i_{k}\in\{0,1\} \\ \sum_{j=1}^{k}i_j\ \text{is even}}}T(a_{i_1},\dots,a_{i_k},\underbrace{u,\dots,u}_{n-k\ \text{copies}})\\
&\qquad-\sum_{\substack{i_1,\dots,i_{k}\in\{0,1\} \\ \sum_{j=1}^{k}i_j\ \text{is odd}}}T(a_{i_1},\dots,a_{i_k},\underbrace{u,\dots,u}_{n-k\ \text{copies}})\\
&\in[c,d],
\end{align*}
where
\[
c:=-\sum_{\substack{i_1,\dots,i_{k}\in\{0,1\} \\ \sum_{j=1}^{k}i_j\ \text{is odd}}}T(a_{i_1},\dots,a_{i_k},\underbrace{b-a,\dots,b-a}_{n-k\ \text{copies}}),
\]
and
\[
d:=\sum_{\substack{i_1,\dots,i_{k}\in\{0,1\} \\ \sum_{j=1}^{k}i_j\ \text{is even}}}(-1)^{\sum_{j=1}^{k}i_j}T(a_{i_1},\dots,a_{i_k},\underbrace{b-a,\dots,b-a}_{n-k\ \text{copies}}).
\]

Boundedness can be explored more generally using the notion of a bornology.

\begin{definition}\label{D: bornology} A \textit{bornology} on a set $X$ is a collection $\mathcal{B}$ of subsets of $X$ such that the following hold:
	\begin{enumerate}
		\item[(i)] If $A, B \in \mathcal{B}$, then $A \cup B \in \mathcal{B}$,
		\item[(ii)] if $B \in \mathcal{B}$ and $A \subseteq B$, then $A \in \mathcal{B}$, and
		\item[(iii)] $\bigcup_{B \in \mathcal{B}} B = X$.
	\end{enumerate}
 We call $(X, \mathcal{B})$ a \textit{bornological space} and often denote it simply by $X$, for short. Moreover, we say that $B\subseteq X$ is \textit{bounded} whenever $B\in\mathcal{B}$. A map $T\colon X\to Y$ between bornological spaces $X$ and $Y$ is called \textit{bounded} if $T(B)$ is bounded whenever $B$ is bounded.
\end{definition}

\begin{definition}
Given bornological spaces $(X,\mathcal{B})$ and $(Y,\mathcal{C})$, we in this paper always endow $X\times Y$ with the \textit{product bornology} $\mathcal{P}$, defined as
\[
\mathcal{P}:=\{D\subseteq X\times Y\ :\ \text{there exists}\ B\in\mathcal{B}, C\in\mathcal{C}\ \text{such that}\ D\subseteq B\times C\}.
\]
We call a vector space $E$ with a bornology a \textit{bornological vector space} if both of the maps $E\times E\ni (f,g)\mapsto f+g\in E$ and $\R\times E\ni(\lambda, f)\mapsto \lambda f\in E$ are bounded. If a bornological vector space $(E,\mathcal{B})$ has the property that the convex hull $\mathrm{co(B)}$ of $B$ is bounded for all $B\in\mathcal{B}$, we say that $(E,\mathcal{B})$ is a \textit{convex bornological space}. Finally, we say that a bornological vector space $(E,\mathcal{B})$ is \textit{separated} if the only bounded vector subspace of $E$ is $\{0\}$. 
\end{definition}

\begin{example}
Let $E$ be an ordered vector space with generating positive cone, and let $\mathcal{B}$ be the collection of all order bounded subsets of $E$. For $B\in\mathcal{B}$ and $a,b\in E$, we for brevity write $a\leq B\leq b$ whenever $a\leq x\leq b$ holds for all $x\in B$. If $A,B\in\mathcal{B}$, then there exist $a,b,c,d\in E$ such that $a\leq A\leq b$ and $c\leq B\leq d$. Since the cone of $E$ is generating, $\{a,c\}$ has a lower bound, say $\ell$, and $\{b,d\}$ has un upper bound, say $u$. It follows that $\ell\leq A\cup B\leq u$. Hence axiom $(i)$ of Definition~\ref{D: bornology} is satisfied. We add that condition $(ii)$ is trivial, and $(iii)$ follows from the fact that for all $x\in E$, we have $x\leq x\leq x$, and so $\{x\}\in\mathcal{B}$. These facts show that $\mathcal{B}$ is a bornology on $E$, which is called the \textit{order bornology}. Furthermore, if $B\in\mathcal{B}$ and $x,y\in\mathcal{B}$ with $a\leq x,y\leq b$, then $2a\leq x+y\leq 2b$, so the vector addition map is bounded. Moreover, if $a\leq x\leq b$ and $\lambda\in\mathbb{R}$, then $c\leq x\leq d$ holds for any $c,d\in E$ for which $c\leq\{\lambda a,\lambda b\}$ and $\{\lambda a, \lambda b\}\leq d$. Thus scalar multiplication is also order bounded, and hence $(E,\mathcal{B})$ is a bornological vector space. 

Additionally, if $B\in\mathcal{B}$ and $x,y\in B$ with $a\leq x,y\leq b$, then for all $t\in[0,1]$, we have
\[
a=ta+(1-t)a\leq tx+(1-t)y\leq tb+(1-t)b=b,
\]
showing that $(E,\mathcal{B})$ is a \textit{convex bornological space}. Finally, suppose $E$ is Archimedean and that $V$ is an order bounded vector subspace of $E$. Then there exist $a,b\in E$ such that, for all $v\in V$, we have $a\leq v\leq b$. Fix $v\in V$. Since $V$ is a vector subspace, $nv\in V$ for each $n\in\N$, and thus $nv\leq b$ holds for all $n\in\N$. Since $E$ has a generating positive cone, there exists $p\in E$ with $0\leq p$ and $b\leq p$. Hence for all $n\in\N$ we have $nv\leq p$. Therefore, we get $v\leq 0$, as $E$ is Archimedean. As a similar argument shows that $-v\leq 0$ as well, we obtain $v=0$, and so $V=\{0\}$. We conclude that $(E,\mathcal{B})$, with the assumption that $E$ is Archimedean, is separated.
\end{example}

\section{Main Results}\label{S: MR}

Let $n\in\N$. Given $x_{1},\dots,x_{n}\in\mathbb{R}$, the $n$th harmonic mean $\mathfrak{H}_n$ is defined as
\[
\mathfrak{H}_n(x_1,\dots,x_n)=\begin{cases}
	\dfrac{n\prod\limits_{i=1}^{n}|x_i|}{\sum_{j=1}^{n}\left(\dfrac{1}{|x_j|}\prod\limits_{i=1}^{n}|x_i|\right)} \quad & x_1,\dots,x_n\neq 0 \\
	0 \quad & \text{else}  \\
\end{cases}.
\]

As a continuous and positively homogeneous function, the harmonic and mean is defined in any $\mathfrak{H}_n$-\textit{complete} Archimedean vector lattice--in particular--any uniformly complete Archimedean vector lattice $E$--using the Archimedean vector lattice functional calculus. Since every Archimedean vector lattice $E$ is trivially $\mathfrak{H}_1$-complete, with $\mathfrak{H}_1(f)=|f|$ for all $f\in E$, we will mostly restrict our interest to $n\in\N\setminus\{1\}$ in this paper.

We next show that the harmonic mean of \textit{any} finite number of variables exists in every $\mathfrak{H}_2$-complete Archimedean vector lattice.

\begin{proposition}\label{P: H_2 cmpte implies H_n cmpte}
If $E$ is an $\mathfrak{H}_2$-complete Archimedean vector lattice, then, for every $n\in\N\setminus\{1\}$, $E$ is $\mathfrak{H}_n$-complete.
\end{proposition}

\begin{proof}
Let $E$ be an $\H_2$-complete Archimedean vector lattice. Put $n\in\N\setminus\{1\}$. We proceed by using the strong form of mathematical induction. To this end, we have that $\H_2(f_1,f_2)$ exists in $E$ for all $f_1,f_2\in E$, by assumption. For the inductive hypothesis, assume that, for every $k\leq n$ and all $f_1,\dots,f_k\in E$, that $\H_k(f_1,\dots,f_k)$ exists in $E$. Next let $\omega$ be a real-valued vector lattice homomorphism defined on the vector sublattice of $E$ generated by $\{f_1,\dots,f_{n+1}\}$. For each $k\in\{1,\dots,n+1\}$, we set $y_k:=|\omega(f_k)|$, for ease of notation. If there exists $k\in\{1,\dots,n+1\}$ such that $y_k=0$, then we have $\mathfrak{H}_{n+1}\bigl(\omega(f_1),\dots,\omega(f_{n+1})\bigr)=0=\omega(0)$, which shows that $h(f_1,\dots, f_{n+1})=0$. Assume next that, for all $k\in\{1,\dots, n+1\}$, we have $y_k\neq 0$. Utilizing \cite[Theorem~3.11]{BusSch} in the sixth and seventh equalities below, we have
\begin{align*}
\H_{n+1}\bigl(\omega(f_1),\dots,\omega(f_{n+1})\bigr)&=\dfrac{(n+1)\prod\limits_{i=1}^{n+1}y_i}{\sum_{j=1}^{n+1}\left(\dfrac{1}{y_j}\prod\limits_{i=1}^{n+1}y_i\right)}\\
&=\dfrac{n+1}{\sum_{\ell=1}^{n+1}\frac{1}{y_\ell}}\\
&=\dfrac{n+1}{n}\cdot\dfrac{n}{\frac{1}{\frac{1}{2}\H_2(y_1,y_2)}+\sum_{\ell=3}^{n+1}\frac{1}{y_\ell}}\\
&=\dfrac{n+1}{n}\H_n\left(\dfrac{1}{2}\H_2(y_1,y_2), y_3,\dots, y_{n+1}\right)\\
&=\dfrac{n+1}{n}\H_n\left(\dfrac{1}{2}\H_2\bigl(\omega(f_1),\omega(f_2)\bigr), \omega(f_3),\dots, \omega(f_{n+1})\right)\\
&=\dfrac{n+1}{n}\H_n\left(\omega\left(\dfrac{1}{2}\H_2\bigl(f_1,f_2\bigr)\right), \omega(f_3),\dots,\omega(f_{n+1})\right)\\
&=\dfrac{n+1}{n}\omega\left(\H_n\left(\dfrac{1}{2}\H_2\bigl(f_1,f_2\bigr),f_3,\dots, f_{n+1}\right)\right)\\
&=\omega\left(\dfrac{n+1}{n}\H_n\left(\dfrac{1}{2}\H_2\bigl(f_1,f_2\bigr),f_3,\dots, f_{n+1}\right)\right).
\end{align*}
It follows from the inductive hypothesis and Definition~\ref{D: fun cal} that $\H_{n+1}(f_1,\dots,f_{n+1})$ is defined in $E$ and that
\[
\H_{n+1}(f_1,\dots,f_{n+1})=\dfrac{n+1}{n}\H_n\left(\dfrac{1}{2}\H_2\bigl(f_1,f_2\bigr),f_3,\dots, f_{n+1}\right).
\]
\end{proof}

Next we provide an alternative formula for the harmonic mean in universally complete vector lattices. The following definition and lemma are required. We note here that the multiplication on any Archimedean $f$-algebra is commutative \cite[Section~142]{Zan2}.

\begin{lemma}\label{L: P(e)}
Let $E$ be an Archimedean $f$-algebra with multiplicative unit $e$, and let $\omega\colon E\to\R$ be a multiplicative map for which $\omega(e)=1$. The following hold.
\begin{itemize}
\item[(i)] If $\P$ is a projection band on $E$, then $\omega\bigl(\P(e)\bigr)=1$ or $\omega\bigl(\P(e)\bigr)=0$.
\item[(ii)] If $f\in E$, and $E$ is universally complete, then we have $\omega(f)=0$ if and only if $\omega(f^\ast)=0$. 
\item[(iii)] If $E$ is universally complete, then $\omega(f^\ast)=\bigl(\omega(f)\bigr)^\ast$ for every $f\in E$.
\end{itemize}
\end{lemma}

\begin{proof}
(i) Let $\P$ be a projection band on $E$. Since $\omega$ is multiplicative, we have
\[
\Bigl(\omega\bigl(\P(e)\bigr)\Bigr)^2=\omega\Bigl(\bigl(\P(e)\bigr)^2\Bigr)=\omega\bigl(\P(e)\bigr).
\]
Thus $\omega\bigl(\P(e)\bigr)$ is a nonnegative idempotent real number and is hence equal to $1$ or $0$.

(ii) Assume $E$ is universally complete, and let $f\in E$.  Note that
\[
f^\ast ff^\ast=\P_f(e)f^\ast=f^\ast
\]
and
\[
ff^\ast f=\P_f(e)f=f,
\]
so that $f$ and $f^\ast$ are von Neumann inverses. Next assume $\omega(f)=0$. Then we have
\[
\omega(f^\ast)=\omega(f^\ast ff^\ast)=\bigl(\omega(f^\ast)\bigr)^2\omega(f)=0.
\]
A symmetric argument shows that if $\omega(f^\ast)=0$, then $\omega(f)=0$.

(iii) Let $f\in E$. We have $\omega(f)\omega(f^\ast)=\omega\bigl(\P_f(e)\bigr)$. By part (i), $\omega\bigl(\P_f(e)\bigr)=1$ or $\omega\bigl(\P_f(e)\bigr)=0$. If $\omega\bigl(\P_f(e)\bigr)=1$, then $\omega(f^\ast)=\bigl(\omega(f)\bigr)^{-1}=\bigl(\omega(f)\bigr)^\ast$. On the other hand, if $\omega\bigl(\P_f(e)\bigr)=0$, then $\omega(f)=0$ or $\omega(f^\ast)=0$. Then by part (ii), we have $\omega(f)=0$ and $\omega(f^\ast)=0$. Hence $\bigl(\omega(f)\bigr)^\ast=0$, and so we have $\omega(f^\ast)=\bigl(\omega(f)\bigr)^\ast$.
\end{proof}

Using Definition~\ref{D: vn inverse} and Lemma~\ref{L: P(e)}, we get another expression for the harmonic mean in universally complete vector lattices.

\begin{proposition}\label{P: HM in univ cmpte}
Let $E$ be a universally complete vector lattice, fix $n\in\mathbb{N}\setminus\{1\}$, and let $f_1,\dots,f_n\in E^+$. Then we have
\begin{equation}\label{eq: HM in univ cmpte}
\H_n(f_1,\dots,f_n)=n\prod_{k=1}^{n}f_k\left(\sum_{j=1}^{n}f_j^\ast\prod\limits_{i=1}^{n}f_i\right)^\ast.
\end{equation}
\end{proposition}

\begin{proof}
Let $\omega$ be a multiplicative vector lattice homomorphism defined on the $f$-subalgebra of $E$ generated by 
\[
\left\{f_1,\dots,f_n,\H_n(f_1,\dots,f_n), f_1^\ast, \dots, f_n^\ast, \left(\sum_{j=1}^{n}f_j^\ast\prod\limits_{i=1}^{n}f_i\right)^\ast\right\}.
\]
Using Lemma~\ref{L: P(e)}~(iii) in the second and fifth identities below, we obtain
\begin{align*}
\H_n\bigl(\omega(f_1),\dots,\omega(f_n)\bigr)&=n\prod_{k=1}^{n}\omega(f_k)\left(\sum_{j=1}^{n}\bigl(\omega(f_j)\bigr)^\ast\prod\limits_{i=1}^{n}\omega(f_i)\right)^\ast\\
&=n\prod_{k=1}^{n}\omega(f_k)\left(\sum_{j=1}^{n}\omega(f_j^\ast)\prod\limits_{i=1}^{n}\omega(f_i)\right)^\ast\\
&=\omega\left(n\prod_{k=1}^{n}f_k\right)\left(\sum_{j=1}^{n}\omega\left(f_j^\ast\prod\limits_{i=1}^{n}f_i\right)\right)^\ast\\
&=\omega\left(n\prod_{k=1}^{n}f_k\right)\left(\omega\left(\sum_{j=1}^{n}f_j^\ast\prod\limits_{i=1}^{n}f_i\right)\right)^\ast\\
&=\omega\left(n\prod_{k=1}^{n}f_k\right)\omega\left(\left(\sum_{j=1}^{n}f_j^\ast\prod\limits_{i=1}^{n}f_i\right)^\ast\right)\\
&=\omega\left(n\prod_{k=1}^{n}f_k\left(\sum_{j=1}^{n}f_j^\ast\prod\limits_{i=1}^{n}f_i\right)^\ast\right).
\end{align*}
Using Definition~\ref{D: fun cal}, we conclude that \eqref{eq: HM in univ cmpte} holds.
\end{proof}

Adding to Proposition~\ref{P: HM in univ cmpte}, we note that the formula \eqref{eq: HM in univ cmpte} is also valid in any $\mathfrak{H}_n$-complete Archimedean semiprime $f$-algebra $E$, though it's interpretation is a bit more complicated. To explain, we can consider the operation $f\mapsto f^\ast$ as a map from $E$ to the universal completion $E^u$ of $E$, where \eqref{eq: HM in univ cmpte} was shown to hold. However, since $E$ is $\mathfrak{H}_n$-complete, we actually have that the right hand side of \eqref{eq: HM in univ cmpte} is an element of $E$.

Next fix $n\in\N\setminus\{1\}$, and let $E$ be an $\H_n$-complete Archimedean vector lattice. Noting that $\H_n$ is concave on $(\mathbb{R}^n)^+$, it follows from Theorem~\ref{T: convex & concave}~(ii) that
\[
\H_n(f_1,\dots,f_n)=n\inf\left\{\sum_{k=1}^na_kf_k\ :\ 0\leq a_1,\dots,a_n\leq 1,\ \sum_{k=1}^n\sqrt{a_k}=1\right\}\ (f_{1},\dots,f_{n}\in E^+).
\]
In particular, we have
\[
\H_2(f,g)=2\inf\{af+bg : 0\leq a,b\leq 1, \sqrt{a}+\sqrt{b}=1\}\qquad (f,g\in E^+).
\]
Therefore, following the definition of square mean closed and geometric mean closed vector lattices in \cite{Az}, we can also discuss harmonic mean closed vector lattices.

\begin{definition}
	An Archimedean vector lattice $E$ is said to be \textit{harmonic mean closed} if, for every $f,g\in E^+$, we have that $\inf\{af+bg : 0\leq a,b\leq 1, \sqrt{a}+\sqrt{b}=1\}$ exists in $E$. In this case, we write
	\[
	f\boxdot g:=2\inf\{af+bg : 0\leq a,b\leq 1, \sqrt{a}+\sqrt{b}=1\}\quad (f,g\in E^+).
	\]
\end{definition}

The following proposition then follows immediately from Theorem~\ref{T: convex & concave}~(ii) and Proposition~\ref{P: H_2 cmpte implies H_n cmpte}.

\begin{proposition}
Let $E$ be an Archimedean vector lattice. The following are equivalent.
\begin{itemize}
	\item[(i)] $E$ is harmonic mean closed,
	\item[(ii)] $E$ is $\mathfrak{H}_2$-complete, and
	\item[(iii)] for every $n\in\N$, $E$ is $\mathfrak{H}_n$-complete.
\end{itemize}
\end{proposition}

We continue this section with the following proposition regarding the harmonic mean, which will aid our proof of Theorem~\ref{T:HM}.

\begin{proposition}\label{P:HM equiv to inf}
Let $E$ be a harmonic mean closed vector lattice, and put $n\in\mathbb{N}\setminus\{1\}$. Then
\[
\bigwedge\limits_{k=1}^nf_k\leq\H_n(f_1,\dots,f_n)\leq n\bigwedge\limits_{k=1}^nf_k
\]
holds for all $f_1,\dots,f_n\in E^+$.
\end{proposition}

\begin{proof}
Let $f_1,\dots,f_n\in E^+$, and let $i\in\lbrace 1,\dots,n\rbrace$ be arbitrary. Then we have

\begin{align*}
\bigwedge\limits_{k=1}^nf_k&=\H_n(1,\dots,1)\left(\bigwedge\limits_{k=1}^nf_k\right)\\
&=n\inf\left\{\sum_{k=1}^na_k\cdot 1\ :\ 0\leq a_1,\dots,a_n\leq 1,\ \sum_{k=1}^n\sqrt{a_k}=1\right\}\left(\bigwedge\limits_{k=1}^nf_k\right)\\
&=n\inf\left\{\sum_{k=1}^na_k\left(\bigwedge\limits_{k=1}^nf_k\right)\ :\ 0\leq a_1,\dots,a_n\leq 1,\ \sum_{k=1}^n\sqrt{a_k}=1\right\}\\
&\leq n\inf\left\{\sum_{k=1}^na_kf_k\ :\ 0\leq a_1,\dots,a_n\leq 1,\ \sum_{k=1}^n\sqrt{a_k}=1\right\}\\
&=\H_n(f_1,\dots,f_n)\\
&\leq nf_i,
\end{align*}
where the last inequality follows from taking $a_i=1$ and $a_j=0$ for all $j\in\{1,\dots,n\}\setminus\{i\}$ in the last infimum above. We conclude that
\[
\bigwedge\limits_{k=1}^nf_k\leq\H_n(f_1,\dots,f_n)\leq n\bigwedge\limits_{k=1}^nf_k.
\]
\end{proof}

We next present a corollary which immediately follows from Proposition~\ref{P:HM equiv to inf}.

\begin{corollary}\label{C:ortho}
Let $E$ be a harmonic mean closed vector lattice, and put $n\in\mathbb{N}\setminus\{1\}$. Then
\[
\H_n(f_1,\dots,f_n)=0
\]
holds whenever $f_1,\dots,f_n\in E^+$ and $f_i\perp f_j$ for some $i,j\in\{1,\dots,n\}$.
\end{corollary}

\begin{notation}
In the theorem below as well as its proof, we at times for brevity will write 
\[
\H(f_i)_{i=1}^n:=\H_n(f_{1},\dots,f_{n}).
\]
We also denote the vector lattice $n$-power of $E$ by $\left(E^{\textcircled{n}},\textcircled{n}\right)$. See \cite{BoBus} for more details.
\end{notation}

We next present our primary result of the paper. It provides characterizations of certain orthogonally additive polynomials completely in terms of the harmonic mean.

\begin{theorem}\label{T:HM}
Let $n\in\mathbb{N}\setminus\{1\}$. Suppose $E$ is an Archimedean vector lattice, $Y$ is a vector space, and $P\colon E\to Y$ is an $n$-homogeneous polynomial.
\begin{itemize}
	\item[(i)] If $E$ is uniformly complete, $Y$ is a separated convex bornological space, and $P$ is bounded, then $P$ is orthogonally additive if and only if 
	\[
	\check{P}(f_{1},\dots,f_{n})=
	\]
	\[
	\frac{1}{n}\Bigl(\check{P}\bigl(\H(f_i)_{i=1}^n,f_2,\dots,f_n\bigr)+\check{P}\bigl(f_{1},\H(f_i)_{i=1}^n,f_3,\dots,f_n\bigr)+\cdots+\check{P}\bigl(f_{1},\dots,f_{n-1},\H(f_i)_{i=1}^n\bigr)\Bigr)
	\]
	holds for every $f_{1},\dots,f_{n}\in E^{+}$.
	\item[(ii)] If $E$ is harmonic mean closed, $Y$ is a uniformly complete Archimedean vector lattice, and $P$ is positive, then $P$ is orthogonally additive if and only if 
	\[
	\check{P}(f_{1},\dots,f_{n})=
	\]
	\[
	\frac{1}{n}\Bigl(\check{P}\bigl(\H(f_i)_{i=1}^n,f_2,\dots,f_n\bigr)+\check{P}\bigl(f_{1},\H(f_i)_{i=1}^n,f_3,\dots,f_n\bigr)+\cdots+\check{P}\bigl(f_{1},\dots,f_{n-1},\H(f_i)_{i=1}^n\bigr)\Bigr)
	\]
	holds for every $f_{1},\dots,f_{n}\in E^{+}$. 
\end{itemize}
\end{theorem}

\begin{proof}
(i) Assume $E$ is uniformly complete, $Y$ is a separated convex bornological space, and $P$ is bounded. Suppose that $P$ is orthogonally additive. Again let $E^u$ denote the universal completion of $E$, and again denote the $f$-algebra multiplication on $E^u$ by juxtaposition. Suppose that $f_1,\dots f_n\in E^+$.

\bigskip
\noindent
\textbf{Step 1.} Given $j\in\{1,...,n\}$, we denote the product
\[
f_1f_2\cdots f_{j-1}f_{j+1}\cdots f_{n-1}f_n
\]
by
\[f_1\cdots\bar{f}_j\cdots f_n,
\]
for short. We claim that in $E^u$ we have
\[
\H_n(f_1,\dots,f_n)\sum_{j=1}^{n}(f_1\cdots\bar{f}_j\cdots f_n)=n(f_1\cdots f_n).
\]
To verify this claim, let $C$ be the Archimedean $f$-subalgebra of $E^u$ generated by
\[
\left\{f_1,\dots,f_n,\H_n(f_1,\dots,f_n)\right\}.
\]
Suppose that $\omega\colon C\rightarrow\mathbb{R}$ is a nonzero multiplicative vector lattice homomorphism. Using \cite[Theorem~3.7]{BusSch} in the second and fourth equalities below, we obtain
\begin{align*}
	\omega&\left(\H_n(f_1,\dots,f_n)\sum_{j=1}^{n}(f_1\cdots\bar{f}_j\cdots f_n)\right)=\omega\left(\H_n(f_1,\dots,f_n)\right)\omega\left(\sum_{j=1}^{n}(f_1\cdots\bar{f}_j\cdots f_n)\right)\\
	&=\H_n\bigl(\omega(f_1),\dots,\omega(f_n)\bigr)\sum_{j=1}^{n}\Bigl(\omega(f_1)\cdots\overline{\omega(f_j)}\cdots\omega(f_n)\Bigr)\\
	&=n(\omega(f_1)\cdots\omega(f_n))\\
	&=\omega\bigl(n(f_1\cdots f_n)\bigr).
\end{align*}
Since the set of all nonzero multiplicative vector lattice homomorphisms $\omega\colon C\rightarrow\mathbb{R}$ separates the points of $C$ (see \cite[Corollary 2.7]{BusdPvR}), we have
\[
\H_n(f_1,\dots,f_n)\sum_{j=1}^{n}(f_1\cdots\bar{f}_j\cdots f_n)=n(f_1\cdots f_n),
\]
as claimed.

\bigskip
\noindent
\textbf{Step 2.} Notice that in $E^u$ we have the elementary identity
\[
\H_n(f_1,\dots,f_n)\sum_{j=1}^{n}(f_1\cdots\bar{f}_j\cdots f_n)=
\]
\[
\H(f_i)_{i=1}^n f_2\cdots f_n+f_{1}\H(f_i)_{i=1}^n f_3\cdots f_n+\cdots+f_{1}\cdots f_{n-1}\H(f_i)_{i=1}^n.
\]

\bigskip
\noindent
\textbf{Step 3.} Arguing as in \cite{Kusa}, we assert that
\[
n\textcircled{n}(f_{1},\dots,f_{n})=
\]
\[
\textcircled{n}\bigl(\H(f_i)_{i=1}^n,f_2,\dots,f_n\bigr)+\textcircled{n}\bigl(f_{1},\H(f_i)_{i=1}^n,f_3,\dots,f_n\bigr)+\cdots+\textcircled{n}\bigl(f_{1},\dots,f_{n-1},\H(f_i)_{i=1}^n\bigr).
\]

Indeed, by \cite[Theorem~4.1]{BoBus}, there exists a uniformly complete vector sublattice $F$ of $E^u$ and a vector lattice isomorphism $i\colon E^{\textcircled{n}}\to F$ such that both
\[
x_1\cdots x_n\in F
\]
and
\[
i\circ\textcircled{n}(x_1,\dots,x_n)=x_1\cdots x_n
\]
hold for all $x_1,\dots,x_n\in E$.

It thus follows from Step 1 and Step 2 that 
\begin{align*}
&n\textcircled{n}(f_{1},\dots,f_{n})=i^{-1}\Bigl(n(f_1\cdots f_n)\Bigr)\\
&=i^{-1}\left(\H_n(f_1,\dots,f_n)\sum_{j=1}^{n}(f_1\cdots\bar{f}_j\cdots f_n)\right)\\
&=i^{-1}\Bigl(\bigl(\H(f_i)_{i=1}^nf_2\cdots f_n\bigr)+\bigl(f_{1}\H(f_i)_{i=1}^n f_3\cdots f_n\bigr)+\cdots+\bigl(f_{1}\cdots f_{n-1}\H(f_i)_{i=1}^n\bigr)\Bigr)\\
&=i^{-1}\bigl(\H(f_i)_{i=1}^nf_2\cdots f_n\bigr)+i^{-1}\bigl(f_{1}\H(f_i)_{i=1}^n f_3\cdots f_n\bigr)+\cdots+i^{-1}\bigl(f_{1}\cdots f_{n-1}\H(f_i)_{i=1}^n\bigr)\\
&=\textcircled{n}\bigl(\H(f_i)_{i=1}^n,f_2,\dots,f_n\bigr)+\textcircled{n}\bigl(f_{1},\H(f_i)_{i=1}^n,f_3,\dots,f_n\bigr)+\cdots+\textcircled{n}\bigl(f_{1},\dots,f_{n-1},\H(f_i)_{i=1}^n\bigr).
\end{align*}

\bigskip
\noindent
\textbf{Step 4.} By \cite[Theorem~2]{Kusa2}, there exists a bounded linear operator $S\colon E^{\textcircled{n}}\to Y$ for which
\[
\check{P}(x_1,\dots,x_n)=S\textcircled{n}(x_1,\dots,x_n)
\]
holds for every $x_1,\dots,x_n\in E$. It thus follows from Step 3 that
\begin{align*}
&n\check{P}(f_1,\dots,f_n)=nS\textcircled{s}(f_1,\dots,f_n)\\
&=S\Bigl(n\textcircled{n}(f_1,\dots,f_n)\Bigr)\\
&=S\Bigl(\textcircled{n}\bigl(\H(f_i)_{i=1}^n,f_2,\dots,f_n\bigr)+\textcircled{n}\bigl(f_{1},\H(f_i)_{i=1}^n,f_3,\dots,f_n\bigr)+\dots+\textcircled{n}\bigl(f_{1},\dots,f_{n-1},\H(f_i)_{i=1}^n\bigr)\Bigr)\\
&=S\textcircled{n}\bigl(\H(f_i)_{i=1}^n,f_2,\dots,f_n\bigr)+S\textcircled{n}\bigl(f_{1},\H(f_i)_{i=1}^n,f_3,\dots,f_n\bigr)+\cdots+S\textcircled{n}\bigl(f_{1},\dots,f_{n-1},\H(f_i)_{i=1}^n\bigr)\\
&=\check{P}\bigl(\H(f_i)_{i=1}^n,f_2,\dots,f_n\bigr)+\check{P}\bigl(f_{1},\H(f_i)_{i=1}^n,f_3,\dots,f_n\bigr)+\cdots+\check{P}\bigl(f_{1},\dots,f_{n-1},\H(f_i)_{i=1}^n\bigr).
\end{align*}

Thus the first implication is proved.

\bigskip
\noindent
\textbf{Step 5.} To prove the converse, suppose that
\[
\check{P}(f_{1},\dots,f_{n})=
\]
\[
\frac{1}{n}\Bigl(\check{P}\bigl(\H(f_i)_{i=1}^n,f_2,\dots,f_n\bigr)+\check{P}\bigl(f_{1},\H(f_i)_{i=1}^n,f_3,\dots,f_n\bigr)+\cdots+\check{P}\bigl(f_{1},\dots,f_{n-1},\H(f_i)_{i=1}^n\bigr)\Bigr)
\]
holds for  every $f_{1},\dots,f_{n}\in E_+$. Let $f,g\in E_+$ with $f\perp g$. It follows immediately from our assumption and Corollary~\ref{C:ortho} that for any $k\in\{1,...,n-1\}$ we have
\[
\check{P}(\underbrace{f,\dots,f}_{k\ \text{copies}},\underbrace{g,\dots,g}_{n-k\ \text{copies}})=0.
\]
Then the binomial theorem gives us
\begin{align*}
P(f+g)=P(f)+P(g)+\sum_{k=1}^{n-1}\binom{n}{k}\check{P}(\underbrace{f,\dots,f}_{k\ \text{copies}},\underbrace{g,\dots,g}_{n-k\ \text{copies}})=P(f)+P(g).
\end{align*}
Hence $P$ is positively orthogonally additive. Finally, from \cite[Theorem~2.3]{Sch} we obtain that $\check{P}$ is orthogonally additive.

\noindent
(ii) Suppose $E$ is harmonic mean closed, $Y$ is a uniformly complete Archimedean vector lattice, and $P$ is positive. Since $P$ is orthogonally additive, we know that $\check{P}$ is orthosymmetric (see e.g. \cite[Theorem~2.8]{Kusa3}). Then, from \cite[Lemma~4.26]{SchThesis}, $\check{P}$ uniquely extends to a positive orthosymmetric $n$-linear map $\check{P}^{\#}\colon\bar{E}^n\to Y$, where $\bar{E}$ is the uniform completion of $E$. It follows from \cite[Theorem~2.8]{Kusa3} again that the positive $n$-homogeneous polynomial $P^{\#}\colon\bar{E}\to Y$ defined by $P^{\#}(f)=\check{P}^{\#}(f,\dots,f)$ is orthogonally additive. Thus, for any $f_1,\dots f_n\in \bar{E}^+$, we from (i) obtain
\[
\check{P}^{\#}(f_{1},\dots,f_{n})=
\]
\begin{align*}
&\frac{1}{n}\Bigl(\check{P}^{\#}\bigl(\H(f_i)_{i=1}^n,f_2,\dots,f_n\bigr)+\check{P}^{\#}\bigl(f_{1},\H(f_i)_{i=1}^n,f_3,\dots,f_n\bigr)+\\
&\qquad\cdots+\check{P}^{\#}\bigl(f_{1},\dots,f_{n-1},\H(f_i)_{i=1}^n\bigr)\Bigr).
\end{align*}
Restricting this equality to $E$, which we can do since $E$ is harmonic mean closed, yields the desired result. Finally, the proof the converse is exactly the same as the corresponding proof in part (i). 
\end{proof}

\section{The QM-AM-GM-HM Inequality}\label{S: QM AM GM HM}
 
As mentioned in Section~\ref{S:intro}, for $n\in\N\setminus\{1\}$, the quadratic mean $\mathfrak{Q}_n$ and the geometric mean $\mathfrak{G}_n$ are defined by
\[
\mathfrak{Q}_n(x_1,\dots,x_n)=\sqrt{\dfrac{1}{n}\sum_{k=1}^{n}x_k^2}\quad (x_1,\dots,x_n\in\mathbb{R})
\]
and
\[
\mathfrak{G}_{n}(x_1,\dots,x_n)=\sqrt[n]{\prod_{k=1}^{n}|x_k|}\quad (x_1,\dots,x_n\in\mathbb{R}).
\]
As continuous positively homogeneous functions, they can be defined in any $\mathfrak{Q}_n$-complete, respectively, $\mathfrak{G}_n$-complete, vector lattice using the Archimedean vector lattice functional calculus. Alternatively, using Theorem~\ref{T: convex & concave}, we obtain the following explicit formulas.

\begin{corollary}\label{C: expl forms}
Fix $n\in\N\setminus\{1\}$. Given a $\mathfrak{Q}_n$-complete, respectively, $\mathfrak{G}_n$-complete, Archimedean vector lattice $E$, $f_1,\dots, f_n\in E^+$, and $\phi:=(\phi_1\dots,\phi_{n-1})\in\left[0,\dfrac{\pi}{2}\right]^{n-1}$, we have
\begin{align*}
&\mathfrak{Q}_n(f_1,\dots,f_n)=\\
&\dfrac{1}{\sqrt{n}}\underset{\phi\in[0,\frac{\pi}{2}]^{n-1}}{\sup}\Biggl\{\cos\phi_{1}f_{1}+\sum\limits_{k=2}^{n-2}\Bigl(\prod\limits_{j=1}^{k-1}\sin\phi_{j}\Bigr)\cos\phi_{k}f_{k}+\Bigl(\prod\limits_{j=1}^{n-2}\sin\phi_{j}\Bigr)\cos\phi_{n-1}f_{n-1}\\
&\qquad+\Bigl(\prod\limits_{j=1}^{n-1}\sin\phi_{j}\Bigr)f_{n}\Biggr\},
\end{align*}
respectively,
\[
\mathfrak{G}_n(f_1,\dots,f_n)=\dfrac{1}{n}\inf\left\{\sum_{k=1}^{n}\theta_kf_k\ :\ \prod_{k=1}^{n}\theta_k=1\right\}.
\]
\end{corollary}

Remarkably, the formula for the geometric mean above makes the AM-GM inequality in Lemma~\ref{L GM-AM ineq}(i) completely trivial. Indeed, this inequality is obtained simply by taking $\theta_1=\cdots=\theta_n=1$ in the set $\left\{\frac{1}{n}\sum_{k=1}^{n}\theta_kf_k\ :\ \prod_{k=1}^{n}\theta_k=1\right\}$. Also observe that Lemma~\ref{L GM-AM ineq}(ii) holds in any Archimedean vector lattice.

\begin{lemma}\label{L GM-AM ineq}[Arithmetic Mean-Geometric Mean Inequality]  Fix $n\in\N\setminus\{1\}$. Let $E$ be an Archimedean vector lattice, and suppose $Y$ is an ordered vector space.
\begin{itemize}
	\item[(i)] If $E$ is $\mathfrak{G}_n$-complete, then $\displaystyle\mathfrak{G}_n(f_1,\dots,f_n)\leq\frac{1}{n}\sum_{k=1}^{n}f_k\quad (f_1,\dots,f_n\in E^+)$.
	\item[(ii)] If $Y$ is a uniformly complete Archimedean vector lattice, and $P\colon E\to Y$ is a positive orthogonally additive $n$-homogeneous polynomial, then we have 
	\[
	\displaystyle\check{P}(f_1,\dots,f_n)\leq P\left(\dfrac{1}{n}\sum_{k=1}^{n}f_k\right)\quad (f_1,\dots,f_n\in E^+).
	\]
\end{itemize}
\end{lemma}

\begin{proof}
We only prove $(ii)$, since $(i)$ is trivial. To this end, assume $Y$ is a uniformly complete Archimedean vector lattice, and suppose $P\colon E\to Y$ is a positive orthogonally additive $n$-homogeneous polynomial. Let $f_1,\dots,f_n\in E^+$. By part (i), we have that
\[
\mathfrak{G}_n(f_1,\dots,f_n)\leq\dfrac{1}{n}\sum_{k=1}^{n}f_k
\]
holds in the uniform completion $\bar{E}$ of $E$ (as $\bar{E}$ is $\mathfrak{G}_n$-complete). Just as in the proof of Theorem~\ref{T:HM}(ii), the maps $\check{P}$ and $P$ uniquely extend to $\check{P}^{\#}$ and $P^{\#}$, respectively, on $\bar{E}$, and $P^{\#}$ is a positive and orthogonally additive $n$-homogeneous polynomial. Utilizing \eqref{eq:GM} and applying $P^{\#}$, which is monotone on the positive cone, to both sides of the previous inequality yields
\[
\check{P}^{\#}(f_1,\dots,f_n)=P^{\#}\bigl(\mathfrak{G}_n(f_1,\dots,f_n)\bigr)\leq P^{\#}\left(\dfrac{1}{n}\sum_{k=1}^{n}f_k\right).
\]
Restricting the map $P^{\#}$ to $E$ in the inequality $\displaystyle\check{P}^{\#}(f_1,\dots,f_n)\leq P^{\#}\left(\dfrac{1}{n}\sum_{k=1}^{n}f_k\right)$ gives us the desired result.
\end{proof}

Turning to the geometric mean-harmonic mean inequality, the following lemma is needed and states that the projection band generated by the harmonic mean of a finite number of elements of a harmonic mean closed Archimedean $f$-algebra with the principal projection property equals the band generated their product.

\begin{lemma}\label{L: bands equal}
Let $E$ be a harmonic mean closed Archimedean $f$-algebra with the principal projection property. Fix $n\in\N\setminus\{1\}$. If $f_1,\dots, f_n\in E^+$, then
\[
B_{\H_n(f_1,\dots, f_n)}=B_{\bigwedge_{k=1}^n f_k}=B_{\prod_{k=1}^{n}f_k}.
\]
\end{lemma}

\begin{proof}
That $B_{\H_n(f_1,\dots, f_n)}=B_{\bigwedge_{k=1}^n f_k}$ follows immediately from Proposition~\ref{P:HM equiv to inf}. Moreover, that $\prod_{k=1}^{n}f_k=0$ if and only if $\bigwedge_{k=1}^n f_k=0$ follows from \cite[Lemma 3.16]{Schtoi}, which states that
\[
\prod_{k=1}^{n}f_k=\prod_{k=1}^{n}\mathcal{M}_k(f_1,\dots,f_n),
\]
where $\bigl\{\mathcal{M}_1(f_1,\dots, f_n),\dots,\mathcal{M}_n(f_1,\dots,f_n)\bigr\}$ is the \textit{total orderization} of $\{f_1,\dots,f_n\}$. It is important to note here that (i) $\mathcal{M}_1(f_1,\dots,f_n)=\bigwedge_{k=1}^nf_k$, and (ii) $\mathcal{M}_k(f_1,\dots,f_n)\leq\mathcal{M}_{k+1}(f_1,\dots,f_n)$\quad ($k\in\{1,\dots, n-1\})$. (See \cite[Definition 2.4 and Lemma 3.16]{Schtoi} for more information.)
\end{proof}

We next present the geometric mean-harmonic mean inequality. The proof is similar to the proof of Theorem~\ref{T:HM} but differs enough that we include most of the details.

\begin{theorem}\label{T: HM-GM ineq}[Geometric Mean-Harmonic Mean Inequality] Let $E$ be an Archimedean vector lattice, and let $Y$ be an ordered vector space. Fix $n\in\N\setminus\{1\}$. Assume $P\colon E\to Y$ is a positive orthogonally additive $n$-homogeneous polynomial.
\begin{itemize}
	\item[(i)] If $E$ is uniformly complete, then we have
	\[
	P\bigl(\H_n(f_1,\dots, f_n)\bigr)\leq\check{P}(f_1,\dots,f_n)\qquad (f_1,\dots,f_n\in E^+).
	\]
	\item[(ii)] If $E$ is harmonic mean closed, and $Y$ is a uniformly complete Archimedean vector lattice, then
	\[
	P\bigl(\H_n(f_1,\dots, f_n)\bigr)\leq\check{P}(f_1,\dots,f_n)\qquad (f_1,\dots,f_n\in E^+).
	\]
\end{itemize}
\end{theorem}

\begin{proof}
(i) We divide this proof into three steps.

\noindent
\textbf{Step 1.} Once more let $E^u$ denote the universal completion of $E$, and suppose that $f_1,\dots f_n\in E^+$. Applying Theorem~\ref{T:HM} to the orthogonally additive polynomial $E\ni f\mapsto f^n\in E^u$ and simultaneously employing Proposition~\ref{P: sm gm} and Lemma~\ref{L GM-AM ineq}(i) in the inequality below, we have

\begin{align*}
	&f_{1}\cdots f_{n}\\
	&=\frac{1}{n}\Bigl(\H(f_i)_{i=1}^nf_2\cdots f_n+f_{1}\H(f_i)_{i=1}^nf_3\cdots f_n\cdots+f_{1}\cdots f_{n-1}\H(f_i)_{i=1}^n\Bigr)\\
	&\geq\sqrt[n]{\left(\eta(f_i)_{i=1}^nf_2\cdots f_n\right)\left(f_{1}\eta(f_i)_{i=1}^nf_3\cdots f_n\right)\cdots\left(f_{1}\cdots f_{n-1}\eta(f_i)_{i=1}^n\right)}\\
	&=\H_n(f_1,\dots,f_n)\sqrt[n]{(f_1\cdots f_n)^{n-1}}.
\end{align*}

Thus we obtain
\[
(f_1\cdots f_n)^n\geq \bigl(\H_n(f_1,\dots,f_n)\bigr)^n(f_1\cdots f_n)^{n-1}.
\]
Hence, letting $e$ be the multiplicative unit of $E^u$, we get
\begin{align*}
	f_1\cdots f_n&=\P_{f_1\cdots f_n}(f_1\cdots f_n)\\
	&=\P_{f_1\cdots f_n}(ef_1\cdots f_n)\\
	&=\P_{f_1\cdots f_n}(e)\P_{f_1\cdots f_n}(f_1\cdots f_n)\\
	&=\P_{f_1\cdots f_n}\bigl((f_1\cdots f_n)(f_1\cdots f_n)^\ast\bigr)\P_{f_1\cdots f_n}(f_1\cdots f_n)\\
	&=\P_{f_1\cdots f_n}\bigl((f_1\cdots f_n)^2(f_1\cdots f_n)^\ast\bigr)\\
	&=(f_1\cdots f_n)^2(f_1\cdots f_n)^\ast\\
	&=\P_{f_1\cdots f_n}\bigl(e(f_1\cdots f_n)^2(f_1\cdots f_n)^\ast\bigr)\\
	&=\P_{f_1\cdots f_n}(e)\P_{f_1\cdots f_n}\bigl((f_1\cdots f_n)^2(f_1\cdots f_n)^\ast\bigr)\\
	&=\P_{f_1\cdots f_n}\bigl((f_1\cdots f_n)(f_1\cdots f_n)^\ast\bigr)\P_{f_1\cdots f_n}\bigl((f_1\cdots f_n)^2(f_1\cdots f_n)^\ast\bigr)\\
	&=\P_{f_1\cdots f_n}\Bigl((f_1\cdots f_n)^3\bigl((f_1\cdots f_n)^\ast\bigr)^2\Bigr)\\
	&=\P_{f_1\cdots f_n}\Bigl((f_1\cdots f_n)^3\bigl((f_1\cdots f_n)^2\bigr)^\ast\Bigr)\\
	&=(f_1\cdots f_n)^3\bigl((f_1\cdots f_n)^2\bigr)^\ast.
\end{align*}

Continuing this pattern, we get
\[
f_1\cdots f_n=(f_1\cdots f_n)^n\bigl((f_1\cdots f_n)^{n-1}\bigr)^\ast.
\]
Therefore, using Lemma~\ref{L: bands equal} in the last identity below, we have
\begin{align*}
f_1\cdots f_n&=(f_1\cdots f_n)^n\bigl((f_1\cdots f_n)^{n-1}\bigr)^\ast\\
&\geq\bigl(\H_n(f_1,\dots,f_n)\bigr)^n(f_1\cdots f_n)^{n-1}\bigl((f_1\cdots f_n)^{n-1}\bigr)^\ast\\
&=\bigl(\H_n(f_1,\dots,f_n)\bigr)^n\, \P_{f_1\cdots f_n}(e)\\
&=\bigl(\H_n(f_1,\dots,f_n)\bigr)^n.
\end{align*}

\bigskip
\noindent
\textbf{Step 2.} We argue that
\[
\textcircled{n}(f_{1},\dots,f_{n})\geq \textcircled{n}\bigl(\H(f_i)_{i=1}^n,\dots, \H(f_i)_{i=1}^n\bigr).
\]
As noted in the proof of Theorem~\ref{T:HM}, there exists a uniformly complete vector sublattice $F$ of $E^u$ and a vector lattice isomorphism $i\colon E^{\textcircled{n}}\to F$ such that both $x_1\cdots x_n\in F$ and $i\circ\textcircled{n}(x_1,\dots,x_n)=x_1\cdots x_n$
hold for all $x_1,\dots,x_n\in E$. It thus follows from Step 1 and the fact that $i^{-1}$ is a positive linear map that 
\begin{align*}
\textcircled{n}(f_{1},\dots,f_{n})&=i^{-1}\bigl((f_1\cdots f_n)\bigr)\geq i^{-1}\Bigl(\bigl(\H(f_1,\dots,f_n)\bigr)^n\Bigr)\\
&=\textcircled{n}\bigl(\H(f_1,\dots,f_n),\dots,\H_n(f_1,\dots, f_n)\bigr).
\end{align*}

\bigskip
\noindent
\textbf{Step 3.} We know there exists a bounded linear operator $S\colon E^{\textcircled{n}}\to Y$ for which
\[
\check{P}(x_1,\dots,x_n)=S\textcircled{n}(x_1,\dots,x_n)
\]
holds for every $x_1,\dots,x_n\in E$. Specifically, we know from \cite[Theorem~2]{Kusa2} that
\[
S(f)=\check{P}(f,|f|,\dots,|f|)\quad (f\in E),
\]
where $E^{\textcircled{n}}$ is identified via vector lattice isomorphism with $E$, where $E$ is endowed with a nonstandard addition and scalar multiplication but with the same ordering (see \cite{Kusa2} for more details.) It thus follows, since $\check{P}$ is a positive $n$-linear map, that $S$ is a positive operator. It hence follows from Step 2 that
\begin{align*}
	\check{P}(f_1,\dots,f_n)&=S\textcircled{n}(f_1,\dots,f_n)\geq S\textcircled{n}\bigl(\H_n(f_1,\dots,f_n),\dots,\H_n(f_1,\dots,f_n)\bigr)\\
	&=\check{P}\bigl(\H_n(f_1,\dots,f_n),\dots,\H_n(f_1,\dots,f_n)\bigr)=P\bigl(\H_n(f_1,\dots, f_n)\bigr).
\end{align*}
(ii) This proof follows from applying part (i) to the unique extension $P^{\#}$ of $P$ to the uniform completion of $E$, just as in the proof of Theorem~\ref{T:HM}~(ii).
\end{proof}

The proof of the quadratic mean-arithmetic mean inequality in $\mathfrak{Q}_n$-complete Archimedean vector lattices also follows from the formulas of Corollary~\ref{C: expl forms}, though there is a little more work to it. It is of special interest to note that Theorem~\ref{T: AM-QM ineq}~(ii) holds for any Archimedean vector lattice--no additional assumptions are needed.

\begin{theorem}\label{T: AM-QM ineq}[Quadratic Mean-Arithmetic Mean Inequality] Fix $n\in\mathbb{N}\setminus\{1\}$, let $E$ be an Archimedean vector lattice, and assume $Y$ is an ordered vector space.
	\begin{itemize}
	\item[(i)] If $E$ is $\mathfrak{Q}_n$-complete, then $\displaystyle\frac{1}{n}\sum_{k=1}^{n}f_k\leq\mathfrak{Q}_n(f_1,\dots,f_n)\quad (f_1,\dots,f_n\in E^+)$.
	\item[(ii)] If $Y$ is a uniformly complete Archimedean vector lattice, and $P\colon E\to Y$ is a positive orthogonally additive $2$-homogeneous polynomial, then we have 
\[
P\left(\dfrac{1}{n}\sum_{k=1}^nf_k\right)\leq\dfrac{1}{n}\sum_{k=1}^nP(f_k)\quad (f_1,\dots,f_n\in E^+).
\]
\end{itemize}
\end{theorem}

\begin{proof}
(i) For $k\in\{1,\dots, n-1\}$, fix $\phi_k:=\arccos\left(\dfrac{1}{\sqrt{n-k+1}}\right)$, and for $\theta_1,\dots,\theta_{n-1}\in\mathbb{R}$, set $\theta:=(\theta_1,\dots,\theta_{n-1})$, for short. Then, for each $k\in\{1,\dots,n-1\}$, we have $\cos\phi_k=\dfrac{1}{\sqrt{n-k+1}}$ and $\sin\phi_k=\dfrac{\sqrt{n-k}}{\sqrt{n-k+1}}$. Therefore, from the telescoping nature of the products below, we get

\begin{align*}
&\sqrt{n}\mathfrak{Q}_n(f_1,\dots, f_n)\\
&=\underset{\theta\in[0,\frac{\pi}{2}]^{n-1}}{\sup}\Biggl\{\cos\theta_{1}f_{1}+\sum\limits_{k=2}^{n-2}\Bigl(\prod\limits_{j=1}^{k-1}\sin\theta_{j}\Bigr)\cos\theta_{k}f_{k}+\Bigl(\prod\limits_{j=1}^{n-2}\sin\theta_{j}\Bigr)\cos\theta_{n-1}f_{n-1}\\
&\qquad+\Bigl(\prod\limits_{j=1}^{n-1}\sin\theta_{j}\Bigr)f_{n}\Biggr\}\\
&\geq\cos\phi_{1}f_{1}+\sum\limits_{k=2}^{n-2}\Bigl(\prod\limits_{j=1}^{k-1}\sin\phi_{j}\Bigr)\cos\phi_{k}f_{k}+\Bigl(\prod\limits_{j=1}^{n-2}\sin\phi_{j}\Bigr)\cos\phi_{n-1}f_{n-1}+\Bigl(\prod\limits_{j=1}^{n-1}\sin\phi_{j}\Bigr)f_{n}\\
&=\dfrac{1}{\sqrt{n}} f_{1}+\sum\limits_{k=2}^{n-2}\Bigl(\prod\limits_{j=1}^{k-1}\dfrac{\sqrt{n-j}}{\sqrt{n-j+1}}\Bigr)\dfrac{1}{\sqrt{n-k+1}}f_{k}+\Bigl(\prod\limits_{j=1}^{n-2}\dfrac{\sqrt{n-j}}{\sqrt{n-j+1}}\Bigr)\dfrac{1}{\sqrt{2}}f_{n-1}\\
&\qquad+\Bigl(\prod\limits_{j=1}^{n-1}\dfrac{\sqrt{n-j}}{\sqrt{n-j+1}}\Bigr)f_{n}\\
&=\dfrac{1}{\sqrt{n}}\sum_{k=1}^nf_k.
\end{align*}
Hence (i) follows.

\noindent
(ii)  Suppose $Y$ is a uniformly complete Archimedean vector lattice, and assume $P\colon E\to Y$ is a positive orthogonally additive $2$-homogeneous polynomial. Consider positive elements $f_1,\dots,f_n\in E^+$. From part (i), we have
\[
\displaystyle\frac{1}{n}\sum_{k=1}^{n}f_k\leq\mathfrak{Q}_n(f_1,\dots,f_n)
\]
in the uniform completion $\bar{E}$ of $E$. Exactly as in the proof of Theorem~\ref{T:HM}(ii), the maps $\check{P}$ and $P$ extend uniquely to $\check{P}^{\#}$ and $P^{\#}$, respectively, on $\bar{E}$, and the $2$-homogeneous polynomial $P^{\#}$ is both positive and orthogonally additive. Applying $P^{\#}$, which is monotone on the positive cone, to both sides of the inequality above, and using \eqref{eq:RMP}, gives us
\[
P^{\#}\left(\dfrac{1}{n}\sum_{k=1}^nf_k\right)\leq P^{\#}\bigl(\mathfrak{Q}_n(f_1,\dots,f_n)\bigr)=\dfrac{1}{n}\sum_{k=1}^nP^{\#}(f_k).
\]
The result (ii) now follows by restricting the map $P^{\#}$ in the inequality
\[
P^{\#}\left(\dfrac{1}{n}\sum_{k=1}^nf_k\right)\leq \dfrac{1}{n}\sum_{k=1}^nP^{\#}(f_k)
\]
to $E$.
\end{proof}

For $n=2$, we can harmonize and summarize Lemma~\ref{L GM-AM ineq}~(ii), Theorem~\ref{T: HM-GM ineq}~(ii), and Theorem~\ref{T: AM-QM ineq}~(ii) in our final result below. It is perhaps remarkable that Theorem~\ref{T: HM-GM-AM-QM}(ii) only needs the domain space to be harmonic mean closed--geometric mean nor quadratic mean completeness are required assumptions.

\begin{theorem}[QM-AM-GM-HM Inequalities]\label{T: HM-GM-AM-QM} Assume that $E$ is a harmonic mean closed Archimedean vector lattice, and let $Y$ be a uniformly complete Archimedean vector lattice. Assume $P\colon E\to Y$ is a positive orthogonally additive $2$-homogeneous polynomial. For $f,g\in E^+$ we have
\[
P(f\boxdot g)\leq\check{P}(f,g)\leq P\left(\dfrac{1}{2}(f+g)\right)\leq\dfrac{1}{2}\bigl(P(f)+P(g)\bigr).
\]
\end{theorem}

\section{Acknowledgments}

\noindent
AI Disclosure: The author utilized OpenAI's ChatGPT to assist in verifying the generalization of spherical coordinates to higher dimensions, explore the classical QM-AM-GM-HM inequalities on the real number line, an investigate some properties of bounded maps between bornological spaces. The core mathematical insights and overall direction of the research were all handled entirely by the author. All AI content was verified by the author, whom assumes all responsibility for the veracity of this manuscript.

\end{document}